\documentclass[11pt]{amsart}
\usepackage{amsfonts,amsmath,mathrsfs,amssymb,amsthm,amscd,latexsym}
\usepackage[usenames]{color}
\usepackage[usenames,dvipsnames,svgnames,table]{xcolor}
\usepackage[all]{xy}
\usepackage{enumerate}
\xyoption{all}
\usepackage{fancyhdr}

\newtheorem{thm}{Theorem}[section]
\newtheorem{lem}[thm]{Lemma}
\newtheorem{cor}[thm]{Corollary}
\newtheorem{prop}[thm]{Proposition}

\theoremstyle{definition}
\newtheorem{rem}[thm]{Remark}
\newtheorem{defn}[thm]{Definition}

\newcommand{\cU}{{\mathcal U}}
\newcommand{\cA}{{\mathcal A}}
\newcommand{\cC}{{\mathcal C}}
\newcommand{\cJC}{{\mathcal {JC}}}
\newcommand{\cY}{{\mathcal Y}}

\newcommand{\lra}{\longrightarrow}

\title[On the dimension of Voisin sets]{On the dimension of Voisin sets in the moduli space of abelian varieties}

\author{E. Colombo} 
\address{Elisabetta Colombo \\ Universit\`a degli Studi di Pavia  \\ Dipartimento di Matematica \\ Via Cesare Saldini 50  \\   20133 Milano, Italy }
\email{elisabetta.colombo@unimi.it}

\author{J.C. Naranjo}
 \address{Juan Carlos Naranjo  \\Universitat de Barcelona  \\ Departament de Matem\`a\-ti\-ques i Inform\`atica \\ Gran Via 585  \\ 08007 Barcelona, Spain  }
 \email{jcnaranjo@ub.edu}

 \author{G.P. Pirola}
 \address{Gian Pietro Pirola  \\ Universit\`a degli Studi di Pavia  \\ Dipartimento di Matematica \\ Via Ferrata 1  \\ 27100 Pavia, Italy  }
 \email{gianpietro.pirola@unipv.it}

 \thanks{
 E. Colombo and G.P. Pirola are members of Gnsaga (INDAM) and are partially supported by PRIN projectModuli spaces and Lie theory (2017), G.P. Pirola is partially supported by MIUR: Dipartimenti di Eccellenza Program (2018-2022) - Dept. of Math. Univ. of Pavia.
  Naranjo was partially supported by the Proyecto de Investigaci\'on MTM2015-65361-P and PID2019-104047GB-100}

\begin{document}

\begin{abstract}{We study the subsets $V_k(A)$ of a complex abelian variety $A$ consisting in the collection of points $x\in A$  such that the zero-cycle $\{x\}-\{0_A\}$ is $k$-nilpotent with respect to the Pontryagin product in the Chow group.  These sets were introduced recently by Voisin and she showed that $\dim V_k(A) \leq k-1$ and $V_k(A)$ is countable for a very general abelian variety of dimension at least $2k-1$.

We study in particular the locus $\mathcal V_{g,2}$ in the moduli space of abelian varieties of dimension $g$  with a fixed polarization, where $V_2(A)$ is positive dimensional. We prove that an irreducible subvariety $\mathcal Y \subset \mathcal V_{g,2}$, $g\ge 3$, such that for a very general $y \in \mathcal Y $ there is a curve in $V_2(A_y)$ generating $A$ satisfies $\dim \mathcal Y \le 2g - 1.$ The hyperelliptic locus shows that this bound is sharp.

}
\end{abstract}

\maketitle

\begin{flushright}
\textit{Dedicated to the memory of our friend Alberto Collino}
\end{flushright}

\section{Introduction}

Claire Voisin in \cite{Vo}  defines the subset $V_k(A)$ of a complex abelian variety $A$ consisting in the collection of points $x\in A$  such that the zero-cycle $\{x\}-\{0_A\}$ is $k$-nilpotent with respect to the Pontryagin product in the Chow group:
\[
 V_k(A):=\{x\in A \mid (\{x\}-\{0_A\})^{*k}=0 \text{ in } CH_0(A)_{\mathbb Q}\}. 
\]
Here we have  denoted by $\{x\}$ the zero-cycle of degree $1$  corresponding to the point $x\in A$. These are naturally defined sets in the sense that they exist in all the abelian varieties, are functorial and move in families. Moreover they are related with the gonality of the abelian variety itself (the minimal gonality of a curve contained in $A$) in a natural way.

We consider the following subsets of the moduli space of abelian varieties of dimension $g$ with a polarization of type $\delta$:
\[
 \mathcal V_{g,k,l}=\{ A \in \mathcal A_g^\delta \mid \dim V_k(A) \ge l \}.
\]
Since the sets $V_k$ are naturally defined, then $\mathcal V_{g,k,l}$ is a union of countably many closed subvarieties of $\mathcal A_g^\delta$. Hence it makes sense to ask about its dimension. 
Put $\mathcal V_{g,k}:=\mathcal V_{g,k,1}$. For an abelian subvariety $B\subset A$ the inclusion $V_k(B)\subset V_k(A)$ holds and a well-known theorem of Bloch implies that $B=V_k(B)$ if $\dim B+1\le k$, hence in this situation $B\subset V_k(A)$.  These are in some sense ``degenerated examples''. In this paper we concentrate in the case $k=2$ and we take care of the non-degenerate case, that is, we will assume that $V_2(A)$ contains some curve generating the abelian variety $A$.

Our main result is:
\vskip 3mm
\begin{thm}\label{m.res}
 Let $g\ge 3$ and consider an irreducible subvariety $\mathcal Y \subset \mathcal V_{g,2}$ such that for a very general  $y\in \mathcal Y$ there is a curve in $V_2(A_y)$ generating $A_y$. Then $\dim \mathcal Y\le 2g-1.$
\end{thm}
This result is sharp due to the fact that the hyperelliptic locus is contained in  $\mathcal V_{g,2}$, see section 2. In fact, the motivation for this study is to understand the geometrical meaning of the positive dimensional components in $V_2$. Our result gives some evidence that there is a link between the existence of hyperelliptic curves in abelian varieties and the fact that $V_2$ is positive dimensional.  We remark that the statement of the theorem (\ref{m.res})  was suggested by the analogous result in \cite{Isog_hyp} concerning hyperelliptic curves.

\vskip 3mm
Section $2$ is devoted to give some preliminaries  and some useful properties of the loci $V_k(A)$ focusing specially on the case $k=2$. A remarkable property is that $V_2(A)$ is the preimage  of the orbit of the image of the origin with respect to rational equivalence in the Kummer variety $Kum(A)$. We also prove the following interesting facts (see Corollaries (\ref{cor1}) and (\ref{cor2})):

\vskip 3mm
\begin{prop} For any abelian variety $A$ the inclusion $$V_k(A)+V_l(A)\subset V_{k+l-1}(A)$$ holds for all $\, 1\le k, l \le g$. Moreover if $C$ is a hyperelliptic curve of genus $g$, and $J(C)$ be its Jacobian variety, then for all $1\le k \le g+1$, we have that $\dim V_{k}(JC)=k-1$ (the maximal possible value).
\end{prop}

The rest of the paper is devoted to the proof of the main theorem. The beginning follows closely the same strategy as in \cite{Isog_hyp} since we reduce to prove the vanishing of a certain adjoint form. The novelty here is that we prove this vanishing by using the action of a family of rationally trivial zero-cycles as in the spirit of  Mumford and Roitman results revisited by Bloch-Srinivas, Voisin and others.

 More precisely, we can assume that there is relative map $f:\cC \lra \cA$ of curves in abelian varieties over a base $\cU$ such that $f(C_y)\subset V_2(A_y)$ generates the abelian variety $A_y$ for all $y\in \cU$. We use the properties of the sets $V_2(A_y)$ to construct a cycle  on the family of abelian varieties which acts trivially on the differential forms. Then, by using deformation of differential forms as in \cite{Isog_hyp} we compute the so-called adjunction form in a generic point of the family. This technique can be traced-back to \cite{collino_pirola}, where  this procedure is introduced for the first time. In section $4$ we assume that, by contradiction, the dimension of the family is $\ge 2g$ and we prove that this implies the existence of a non-trivial adjoint form. This contradicts the results on section 3.

\textbf{Acknowledgments:} We warmly thank Olivier Martin for his careful reading of the paper, and the referee for valuable suggestions that have simplified and clarified the paper.

\section{Preliminaires on the subsets $V_k$ of an abelian variety $A$}

\subsection{On the dimension of $V_k$}

The most part of this subsection comes from \cite{Vo}, where the sets $V_k(A)$ appear for the first time.

\begin{defn}
 Let $A$ be an abelian variety and denote by $CH_0(A)$ the Chow group of zero-cycles in $A$ with rational coefficients. We also denote by $*$ the Pontryagin product in the Chow group. Given a point $x \in A$, we put $\{x\}$ for the class of $x$ in $CH_0(A)$. Then we define the Voisin sets (cf. \cite{Vo}):
 \[
  V_k=V_k(A):=\{x\in A \mid (\{x\}-\{0_A\})^{*k}=0\}.
 \]
\end{defn}

It is known that the set $V_k$ is a countable union of closed subvarieties of $A$. A typical way to obtain points in $V_k$  is given by the following Proposition:
\vskip 3mm
\begin{prop}(\cite[Prop. (1.9)]{Vo})
\label{Vo_Be}
Assume that $\{x_1\}+\dots +\{x_k\}=k\{0_A\}$ in $CH_0(A)$ for some points $x_i\in A$, then for all $i$ we have $x_i\in V_k$.
\end{prop}

Since  the orbits w.r.t. rational equivalence  $\vert k \{0_A\}\vert$ are hard to compute there are only a few examples of positive dimensional components in $V_k(A)$ that we can construct from this proposition. The simplest instance of this comes from a $k$-gonal curve $C$ contained in $A$ such that there is a totally ramified point $p$ for the degree $k$ map $f:C\lra \mathbb  P^1$. Translating we can assume that $p$ is the origin $0_A$ and then the fibers of $f$ provide a $1$-dimensional component in the symmetric product, $k$ times, of $A$. Therefore, by the proposition above, we obtain that $C$ is contained in $V_k(A)$.
Observe that the sets $V_k$ are invariant under isogenies, hence  these positive dimensional components also appear in many other abelian varieties. In particular for any integer $n$, we have that $n_* V_2(A)\subset V_2(A)$. 

\begin{rem}\label{remarks_on_V_k}
We have the following properties:
\begin{enumerate}
 \item [a)] All the abelian varieties containing hyperelliptic cur\-ves have positive dimensional com\-po\-nents in $V_2$.
 \item [b)] A very well known theorem of Bloch (see \cite{bloch}) implies that $$V_{g+1}(A)=A,$$ hence the natural filtration
 \[
  V_1(A)  \subset V_2(A)\subset \ldots \subset V_g(A) \subset V_{g+1}(A)=A
 \]
 has at most $g$ steps.
It is natural to ask what is the behaviour of the dimension of $V_k(A)$, with $k\le g$, for very general abelian varieties, and which geometric properties of $A$ codify these sets. 
 \item [c)] Assume that $B\subset A$ is an abelian subvariety, then $V_k(B)\subset V_k(A)$. In particular, if $k\ge \dim B+1$, then $B\subset V_k(A)$. For instance: all the elliptic curves in $A$ passing through the origin are contained in $V_2(A)$. 
 \item [d)] Let $C$ be a smooth quartic plane curve and let $p$ be a flex point with tangent $t$, then $t\cdot C=3p+q$ and the projection from $q$ provides a collection of zero-cycles of degree $3$ in $JC$ rationally equivalent to $3\{0_{JC}\}$, then $C\subset V_3(JC)$. Using isogenies we get that there are in fact a countably number of curves in $V_3(A)$ for a very general abelian variety of dimension $3$.  
\end{enumerate}
\end{rem}

The following is proved in Theorem (0.8) of \cite{Vo} by using some ideas from \cite{Mumford} and improving the techniques of \cite{Kummer}:
\vskip 3mm
\begin{thm}\label{dim_V_k}  Let $A$ be an abelian variety of dimension $g$. Then:
\begin{enumerate} 
 \item [a)] $\dim V_k(A) \le k-1$.
 \item [b)] If $A$ is very general and $g\ge 2k-1$ we have that $\dim V_k(A)=0$.
\end{enumerate}
\end{thm}

In the specific case of $V_2(A)$ we have the following properties.

\begin{prop}\label{properties_V_2} Let $A$ be an abelian variety and let $Kum(A)$ be its Kummer variety.
 \begin{enumerate}
  \item [a)] We have the equality $V_2(A)=\{x\in A \mid \{x\}+\{-x\}=2\{0_A\}\}.$
  \item [b)] Let   $\alpha : A\lra Kum(A)$ be the quotient map. Then 
  $$
  V_2(A)=\alpha^{-1}(\{y\in Kum(A) \mid \{y\} \sim_{rat} \{\alpha (0_A)\} \}).
  $$
   \end{enumerate}
\end{prop}
\begin{proof}
Part a) follows from the observation that
\[
 (\{x\}-\{0_A\})^{*2}=\{2x\}-2\{x\}+\{0_A\}=0
\]
is equivalent, translating with $-x$, to $\{x\}+\{-x\}=2\{0_A\}$.

To prove b) we first see that $x\in V_2(A)$ if and only if $\{ \alpha (x) \} \sim_{rat} \{\alpha (0_A)\}$. Indeed, assume that $\{x\}+\{-x\}=2\{0_A\}$, applying $\alpha$ we get that $2 \{\alpha(x)\}\sim_{rat} 2\{\alpha ( 0_A)\}$. Since $Alb({Kum(A)})=0$ the Chow group has no torsion (see \cite{Roitman_tors})  therefore $\{\alpha(x)\}\sim_{rat} \{\alpha (0_A)\}$. In the opposite direction, if $\{\alpha (x)\}\sim_{rat} \{\alpha (0_A)\}$ we apply $\alpha^*$ at the level of Chow groups and we obtain that $x\in V_2(A)$. Hence $V_2(A)$ is the pre-image by $\alpha $ of the points rationally equivalent to $\alpha (0_A)$.
\end{proof}

\subsection{Relation with the Chow ring}

In this part we collect some computations on $0$-cycles on abelian varieties which are more or less implicit in \cite{beauville_quelques}, \cite{beauville} and \cite{Vo}.

Let us recall first some facts on the Chow group (with rational coefficients) of an abelian variety $A$ of dimension $g$ which are proved in \cite{beauville}. Let us define the subgroups:
\[
 CH^g_s(A):=\{z\in CH^g(A)_{\mathbb Q} \mid k_*(z)=k^s z, \quad \forall k\in\mathbb Z \}.
\]
Then:
\[
 CH^g(A)_{\mathbb Q}=CH^g_0(A) \oplus CH^g_1(A) \oplus \ldots \oplus CH^g_g(A).
\]
Moreover $CH^g_0(A)=\mathbb Q \{0_A\}$ and $I=\bigoplus _{s\ge 1} CH^g_s(A)$ is the ideal, with respect to the Pontryagin product, of the zero-cycles of degree $0$. It is known that $I^{*\, r}=  \bigoplus _{s\ge r} CH^g_s(A)$ and that $I^{*\, 2}$ is the kernel of the albanese map tensored with $\mathbb Q$:
\[
 I \lra A_{\mathbb Q},
\]
sending a zero cycle $\sum n_i\{a_i\}$ to the sum $\sum n_i \, a_i$ in $A$. Another useful property is that $CH^g_s(A)*CH^g_t(A)=CH^g_{s+t}(A)$. 

We point out that the filtration $V_1(A)\subset V_2(A) \subset \ldots \subset A$ is, in some sense, induced by the filtration $\ldots \subset I^{* \, 2}\subset  I \subset  CH^g(A)_{\mathbb Q}$. Indeed, given a point $x\in A$ we use the notation:
\begin{equation}\label{decomposition}
 \{x\}=\{0_A\}+x_1+\ldots +x_g, \qquad x_i \in CH^g_i(A).
\end{equation}
Then we have:
\begin{prop}\label{V_k_vs_Chow} 
 For all $x\in A$, $x$ belongs to $V_k(A)$ if and only if $$x_{k}=\ldots =x_g=0.$$ In particular $x \in V_2(A)$ if and only if $\{x\}-\{0_A\}\in CH^g_1(A)$.
\end{prop}
\begin{proof}
 We apply to (\ref{decomposition}) the multiplication by $l$ in $A$:
 \[
  \{l x\}=\{0_A\}+l x_1+\ldots + l^g x_g.
 \]
 Using this we have:
 \[
  \begin{aligned}
   (\{x\}-\{0_A\})^{* k}&= \sum_{i=0}^k (-1)^i\binom{k}{i}\{(k-i)x\}=\\
                        &= \sum_{i=0}^k (-1)^i\binom{k}{i}(\{0_A\}+(k-i)x_1+\ldots +(k-i)^g x_g)= \\
                        &= \sum_{i=0}^k (-1)^i\binom{k}{i} \{0_A\}+ \sum_{i=0}^k (-1)^i\binom{k}{i} (k-i)x_1+\ldots \\
                        &+ \sum_{i=0}^k (-1)^i\binom{k}{i}(k-i)^g x_g.
  \end{aligned}
 \]
 Now we use the following formulas (see the proof of Lemma 3.3 in \cite{Vo} or prove them by induction):
\[
 \begin{aligned}
  \sum_{i=0}^k (-1)^i\binom{k}{i}(k-i)^l=& 0 \qquad \text{ if }\, l<k \\
  \sum_{i=0}^k (-1)^i\binom{k}{i}(k-i)^k=& k!
 \end{aligned}
\]
Therefore we obtain that:
\[
  (\{x\}-\{0_A\})^{* k}=k! x_k+\ldots
\]
and similarly $ (\{x\}-\{0_A\})^{* l}=l! x_l+\ldots $ for any $l\ge k$. Hence $x\in V_k(A)$ if and only if $x\in V_l(A)$ for all $l\ge k$ if and only if $x_k=\ldots =x_g=0$.
\end{proof}

 We have several consequences of this Proposition and of its proof:
 
\begin{cor}
With the same notations we have that $x_k=\frac 1 {k!}x_1^k$. Hence $\{x\}=exp(x_1).$ Moreover $x\in V_k$ if and only if $x_k=x_1^k=0$. 
\end{cor}
\begin{proof}
 We have seen along the proof of the Proposition that 
 \[
 (\{x\}-\{0_A\})^{* k}=k! x_k+\ldots \text{ higher degree terms}.
 \] 
 Computing directly we get 
 that 
\[
 (\{x\}-\{0_A\})^{* k}=(x_1+\ldots +x_k)^{*k}=x_1^k+\ldots \text{higher degree terms}.
 \]
 By comparing both formulas we obtain the equality.
 \end{proof}

 \begin{rem}
 Notice that our computations are more or less contained in \cite{beauville_quelques}. Indeed define as in section $4$ of loc. cit. the map
 \[
  \gamma: A \lra I, \qquad a \mapsto \{0\} -\{a\} + \frac 12  (\{0\} -\{a\})^{*2} +\frac 13  (\{0\} -\{a\})^{*3}+\ldots 
 \]
this is a morphims of groups. Then, with our notations, $\gamma (x)= -x_1$. In particular the image of $\gamma $ belongs to $CH^g_1(A)$.
 \end{rem}

 \begin{cor}\label{points_V_2}
  Let $a, b \in A$ be two points such that $$n \{a\} + m\{b\}=(n+m)\{0_A\}$$ for some integers $1\le n,m$. Then $a,b \in V_2(A)$. 
 \end{cor}
\begin{proof}
Decomposing as before:
\[
 \{a\}=\{0\}+a_1+\frac 12 a_1^2+\ldots \qquad  \{b\}=\{0\}+b_1+\frac 12 b_1^2+\ldots
\]
the equality of the statement implies that $n a_1+m b_1=0=n a_1^2+m b_1^2$. Then $b_1=-\frac nm a_1$ and thus $n a_1^2+ \frac {n^2}{m^2} a_1^2=0$, so $a_1^2=b_1^2=0$ and $a,b \in V_2(A)$.  
\end{proof}

\begin{cor}
 Let $\varphi:A\lra B$ be an isogeny, then $\varphi^{-1}(V_k(B))\subset V_k(A)$. In particular $\varphi (V_k(A))=V_k(B)$ and $\varphi ^{-1}(V_k(B))=V_k(A)$.
\end{cor}
\begin{proof}
Since we work with Chow groups with rational coefficients it is clear that  for an integer $n\neq 0$ the map $n_*:CH^g_k(A)\lra CH^g_k(A)$ is bijective. Let $\psi : B\lra A$ be an isogeny such that $\psi \circ \varphi =n$, we deduce that $\varphi _*: CH^g_k(A) \lra CH^g_k(B)$ is injective. Let $x\in \varphi^{-1}(V_k(B))$ and set $\{x\}=\{0_A\}+x_1+\ldots + x_g$. By hypothesis 
\[
\varphi (\{x\}) =\{0_B\} +\varphi _*(x_1)+\ldots +\varphi _*(x_g) \in V_k(B). 
\]
Hence $\varphi_* (x_k)=0$ and $x_k=0$. Therefore $x \in V_k(A)$ and we are done.
\end{proof}

\begin{rem} This corollary also follows from the definition of $V_k$ and the fact that $\varphi _*$ is an isomorphism modulo torsion on Chow groups compatible with the Pontryagin product.
 \end{rem}

Another interesting consequence of this characterization is the following property:
\begin{cor}\label{cor1}
 For any $0\le k,l \le g$ we have that $$V_k(A)+V_l(A)\subset V_{k+l-1}(A).$$
\end{cor}
\begin{proof}
 Let $x\in V_k(A)$, $y\in V_l(A)$. Then $\{x\}=\{0_A\}+x_1+\ldots +x_{k-1}$ and $\{y\}=\{0_A\}+y_1+\ldots +y_{l-1}$. Since $x_i*y_j \in CH^g_{i+j}(A)$ we obtain 
 \[
 \{x+y\}=  \{x\}*\{y\}=\{0_A\}+(x_1+y_1)+ (x_2+x_1*y_1+y_2)+\ldots+ x_{k-1}*y_{l-1}.
 \]
Thus $x+y\in V_{k+l-1}(A)$.
\end{proof}

As an application we have:
\begin{cor}\label{cor2}
 Let $C$ be a hyperelliptic curve of genus $g$. Then 
 \[
\dim V_k(JC)=k-1
\]
for $1\le k\le g$, that is, the maximal possible dimension is attained.
\end{cor}
\begin{proof}
 Choosing a Weiestrass point to define the Abel-Jacobi map we can assume that the curve $C$ is contained in $V_2(JC)$, using inductively the previous Corollary, we have that 
 \[
 C+\stackrel{(k-1)}{\ldots} +C=W^0_{k-1}(C)\subset V_{k}(JC). 
 \]
 \end{proof}

For instance, for the Jacobian of a genus $3$ curve $C$  we have, in the hyperelliptic case, that $\dim V_2(JC)=1$ and $\dim V_3(JC)=2$. If instead  $C$ is a generic quartic plane curve we have that $\dim V_2(JC)=0$ by Theorem (\ref{dim_V_k}) and $\dim V_3(JC)\ge 1$ by  Remark (\ref{remarks_on_V_k},d). Denoting as in that remark $p$ a flex point and $q$ its residual point, we have also the following: for any $x\in C$ such that the tangent line to $C$ in $x$ goes through $q$, $x\in V_2(C)$ (we identify $C$ with the Abel-Jacobi image in $JC$ using $p$). Indeed: there exists a $y\in C$ with $2x+y+q\sim 3p+q$, hence in $JC$ there is a relation of the form $2\{x\}+\{y\}=3\{0\}$ and then Corollary (\ref{points_V_2}) implies that $\{x\}, \{y\} \in V_2(JC)$. 

Assume now that $C$ is a quartic plane with a hyperflex, that is a point $p$ such that $\mathcal O_C(1)\cong \mathcal O_C (4p)$. This condition defines a divisor in $\mathcal M_3$. Embed $C$ in its Jacobian using $p$ as base-point. Then for any bitangent $2x+2y$ we have $\{x\}, \{y\} \in V_2(JC)$. Also, with the same argument, for a standard flex $q$ we have $\{q\}\in V_2(JC)$. Everything suggests that 
the points in ``$C\cap V_2(JC)$'' could have some geometrial meaning. 
Notice that this leaves open the question whether  the dimension of $V_3(JC)$ is $1$ or $2$ for a generic quartic plane curve $C$.

\vskip 3mm
\section{A family of zero-cycles and the action on differential forms}

In this section we begin the proof of the main theorem.
 We proceed by contradiction, hence we assume that there exists an irreducible component $\mathcal Y$ of $\mathcal V_{g,2}$  of dimension $\ge 2g$. By \ref{dim_V_k} we have $\dim V_2(A)\le 1$, for any abelian variety $A$. Hence $V_2(A_y)$ contains curves for all $y\in \mathcal Y$ and,  by hypothesis,  at least one of these curves generates the abelian variety. By a standard argument (involving the properness and countability of relative Chow varieties and the existence of universal families of abelian varieties up to base change) we can assume the existence of the following diagram:
 \[
 \xymatrix@C=1.pc@R=1.8pc{
\mathcal {C}  \ar[rd] \ar[rr]^ f && \mathcal {A} \ar[ld]^{\pi } \\
&\mathcal U, 
}
\]
 where the parameter space $\cU$ comes equipped with a generically finite map $\Phi: \cU \lra \cY $ such that $\Phi(y)$ is the isomorphism class of $A_y$ for all $y\in \mathcal U$. Moreover $f_y:C_y \lra A_y$ is the normalization map of an irreducible curve $f_y(C_y)$ contained in $V_2(A_{y})$, and generating $A_y$, followed by the inclusion. We can also assume that $\cC \lra \cU$ has a section and then that $f$ induces a map of families of abelian varieties $F:\cJC \lra \mathcal A$ over $\mathcal U$.  

To start with we pull-back the families of curves and abelian varieties to $\cC$ itself: 
\[
 \xymatrix@C=1.pc@R=1.8pc{
\mathcal A_{\cC}  \ar[d]_{\pi_{\cC }} \ar[rr] && \mathcal {A} \ar[d]^\pi\\
\cC \ar[rr] && \cU.
}
\]

Now we define a family of zero-cycles in $\mathcal A_{\cC }$ parametrized by $\cC$. Let $s_+: \cC \lra \cA_{\cC}$ be the section given by the maps $(id_{\cC},f)$:

\[
 \xymatrix@C=1.pc@R=1.8pc{
\cC \ar@/_/[ddr]_{id_{\cC}} \ar@/^/[drr]^f 
   \ar[dr]^{s_+}  \\            
  & \cA_{\cC}  \ar[d]^{\pi_{\cC}} \ar[r]  & \cA \ar[d]^{\pi}       \\
  & \cC \ar[r]   & \cU             
  }
\]

Put $\mathcal Z^+:=s_+(\cC)$. Analogously, by considering  $-1_{\mathcal A}\circ f$, where  $-1_{\mathcal A}  $ is the relative $-1$ map on the family of abelian varieties 
we define a section $s_-:\cC \lra \cA_{\mathcal C}$ of $\pi_{\mathcal C}$ and a cycle $\mathcal Z^-:=s_-(\cC)$. Finally the zero section $0_{\mathcal A}$ induces a section $s_0$ and a cycle $\mathcal Z_0$. Set $\mathcal Z=\mathcal Z^++\mathcal Z^- - 2 \mathcal Z_0$, a cycle on $\mathcal A_{\cC}$.  In fact $\mathcal Z^++\mathcal Z^-, 2 \mathcal Z_0 \subset Sym^2 \mathcal A_{\cC}$.  
Observe that  $\mathcal Z_{t}$, $t\in \cC$ is the $0$-cycle 
\[
\mathcal Z_t=\{f(t)\}+\{-f(t)\}-2\{0_{A_y}\}
\]
on $A_y$, where $\pi(t)=y$. Since   $f(t)\in f(C_y)\subset V_2(A_y)$ we have that $\mathcal Z_{t}\sim _{rat}0$ in $A_y$ (see Proposition \ref{properties_V_2}).

We are interested in an infinitesimal deformation of a curve $C_y$ for a general $y\in \mathcal U$. Thus, let us denote  by $\Delta $ the spectrum of the ring of the dual numbers $Spec \, \mathbb C[\varepsilon ]/(\varepsilon ^2)$. We consider a tangent vector $\xi \in T_{\mathcal U}(y)$ and we take a smooth quasi projective curve $B\subset {\mathcal U}$ passing through $y$ and with  $\xi \in T_B(y)$. This induces the  maps
\[
\alpha_{\xi }:\Delta \lra B \lra  \mathcal U.                                                                                                                                                                                                                                                                                   \]
We pull-back to $B$ and to $\Delta $ the families of curves and abelian varieties and the cycle $\mathcal Z$, thus we have

\[
\xymatrix@C=1.pc@R=1.8pc{
\mathcal A_{\cC_{\Delta }}  \ar[d]_{\pi_{\Delta}} \ar[rr] && \mathcal A_{\cC_B} \ar[d]^\pi \ar[rr] && \mathcal A_{\cC} \ar[d]^{\pi_{\mathcal C}}\\
\mathcal \cC_{\Delta }  \ar[d] \ar[rr] && \mathcal C_B \ar[d] \ar[rr] && \mathcal C \ar[d] \\
\Delta \ar[rr]^{\alpha_{\xi}} && B \ar[rr]  && \cU,
}
\]
and   $\mathcal Z_{\Delta }$ (resp.  $\mathcal Z_{B }$) denotes the pull-back to $\mathcal A_{\mathcal C_{\Delta}}$ (resp. to
 $\mathcal A_{\mathcal C_{B}}$) of the cycle $\mathcal Z$.

The cycle $\mathcal Z_B$ determines a cohomological class  $[\mathcal Z_B]\in H^g(\mathcal A_{\cC_B},\Omega^g_{\mathcal A_{\cC_B}})$ and its restriction a class $[\mathcal Z_B]_t$ in
  $H^g(A_{t}, \Omega ^g_{\mathcal A_{\cC_B} \vert A_{ t} } )$, where $\pi(t)=y$.

 The following is well-known by the experts  and is a version of the classical results by Mumford and Roitman on zero-cycles (see \cite[Lemma 2.2]{voisin_AG} and also \cite{Bloch_Srinivas}): 
\begin{prop}\label{Bloch-Srinivas}
 If for any $t\in \cC_B$, the restricted cycle $\mathcal Z_t$ is rationally equivalent to $0$, then there is a dense Zariski open set $\mathcal V\subset \cC_B$ such that 
 $[\mathcal Z_{\mathcal V}]=0$ in $H^g(\mathcal A_{\mathcal V}, \Omega ^g_{\mathcal A_{\mathcal V}})$. In particular for all $t\in \mathcal V$ we have $[\mathcal Z_{\mathcal V}]_t=[\mathcal Z_{B}]_t =0$ in $H^g( A_{t}, \Omega ^g_{\mathcal A_{\mathcal V}\vert A_t})$.
\end{prop}

Now we look at the action of $\mathcal Z_{\Delta }$ in the space of  differential forms on the infinitesimal family of abelian varieties.  
This works as follows.  

Let $\mathcal A_{C_y}=A_y\times C_y$ be the restriction of  $\mathcal A_{C_{\Delta }}$ over $C_y$.
 Consider  $\Omega \in H^0(\mathcal A_{C_y}, \Omega^2_{\mathcal A_{\cC_ {\Delta}}|\mathcal A_{C_y}}) $ 
and define:
\begin{equation}\label{def_action}
 \mathcal Z_{\Delta }^*(\Omega )=(s_+)^*(\Omega)+(s_-)^*(\Omega )-2\, s_0^*(\Omega),
\end{equation}
which belongs to $H^0(  C_y, \Omega^2_{\mathcal C_{\Delta }|C_y}). $
Then we have by (\ref{Bloch-Srinivas}) that this action is trivial which gives the vanishing $ \mathcal Z_{\Delta }^*(\Omega )=0$.

Consider also the family of abelian varieties on $\Delta $:
\[
\xymatrix@C=1.pc@R=1.8pc{
\mathcal \cA_{\Delta }  \ar[d] \ar[rr] && \mathcal A \ar[d] \\
\Delta \ar[rr] && \cU.
}
\]

Notice that there is a natural map $\mathcal A_{\cC_{\Delta } }
\xrightarrow{h} \mathcal A_{\Delta}$ and that the composition:
\[
  \cC_{\Delta } \xrightarrow{s_+}  \mathcal A_{\cC_{\Delta } }
\xrightarrow{ h} \mathcal A_{\Delta}
\]
is simply the original family of curves $f_{\Delta}: \cC_{\Delta } \lra \mathcal A_{\Delta }$ over $\Delta $.
  Therefore 
for  any  form $\Omega' \in H^0(A_y,\Omega^2_{\mathcal A_{\Delta} \vert A_y}) $, denoting $\Omega = h^*(\Omega ')$, we have that 
\[
f_{\Delta}^*(\Omega')=s_+^*(\Omega).
 \]

An almost identical computation can be done with $s_-$ since $-1_{\mathcal A_{\cC_{\Delta}}}$ acts trivially on the $(2,0)$-forms. Finally if $\Omega' \in H^0(A_y,\Omega^2_{\mathcal A_{\Delta} \vert A_y})^0 \subset H^0(A_y,\Omega^2_{\mathcal A_{\Delta} \vert A_y})$, is a form vanishing at the origin, then 
$s^*_0(h^*(\Omega ') )=0$. 

These considerations combined with (\ref{def_action}) gives that for any form $\Omega '\in H^0(A_y , \Omega ^2 _{\cA_{\Delta }|A_y})^0$:
\[
\mathcal Z_{\Delta}^* h^*(\Omega')=2f_{\Delta}^*(\Omega').
\]
Then, using the vanishing of $\mathcal Z^*(\Omega)$, this implies:
\begin{prop} \label{vanishing_Delta}  With the above notations, 
\[
 f_{\Delta}^*:H^0(A_y,\Omega^2_{\mathcal A_{\Delta} \vert A_y})^0 \longrightarrow   H^0(C_y,\Omega^2_{\cC_\Delta \vert C_y})
 \cong H^0(C_y,\omega _{C_y})
\]
is the zero map.
 \end{prop}

We will see in the next section that, if the dimension of the family is $\ge 2g$, then for a generic point of $\mathcal U$ and a convenient infinitesimal deformation  there is a form in $ H^0(A_y,\Omega^2_{\mathcal A_{\Delta } \vert A_y})^0$ with non-trivial image in $H^0(C_y,\Omega^2_{\cC_\Delta \vert C_y})$. This contradicts the Proposition.

\vskip 3mm
\section{The geometry of the adjoint form and end of the proof}

In this section we end the proof of the main Theorem. Assuming that the dimension of the family is $\ge 2g$, we will  find a contradiction with Proposition (\ref{vanishing_Delta}).  
 
As at the beginning of section $3$  we can assume the existence of the following diagram:
 \[
 \xymatrix@C=1.pc@R=1.8pc{
\mathcal {C}  \ar[rd] \ar[rr]^ f && \mathcal {A} \ar[ld]^{\pi } \\
&\mathcal U, 
}
\]
 where the parameter space $\cU$ comes equipped with a generically finite map $\Phi: \cU \lra \cY$. 
 We fix a generic point $y$ in $\cU$ and we denote by $\mathbb T$ the tangent space of $\cU$ at $y$. 
 Observe that $\mathbb T \hookrightarrow Sym^2 H^{1,0}(A_y)^\ast$. Moreover the surjective map $F_y$ induces an inclusion of $W_y:=H^{1,0}(A_y)$ in 
 $H^0(C_y,\omega_{C_y})$. Let $D$ be the base locus of the linear system generated by $W_y$, therefore $W_y\subset H^{0}(C_y,\omega_{C_y}(-D_y))$. 
 Lemma 3.1 in \cite{Isog_hyp} states that for a generic two dimensional subspace $E$ of $W_y$ the base locus of the pencil attached to $E$ is still 
 $D_y$. As in the proof of Theorem 1.4 \cite{Isog_hyp} we consider the map
  sending $\xi \in \mathbb T$, seen as a symmetric map  $\, \cdot \xi:\, W_y=H^{1,0}(A_y)\lra H^{1,0}(A_y)^*$, to its restriction to $E$. 
Let $E_0$ be a complementary of $E$ in $W_y$.   Then we have an element in $E^*\otimes E^* +E^*\otimes E_0^*$ which, by the symmetry, belongs to $Sym^2 E^* + E^* \otimes E_0^*$. 
This last space has dimension $3+2(g-2)=2g-1$. Since $\dim \mathcal Y \ge 2g$ we get that the linear map
\[
\mathbb T \lra Sym^2 E^* + E^* \otimes E_0^*
\]
sending $\xi$ to $\delta_{\xi \mid E} $ has non trivial Kernel. Therefore  we conclude the following:
  
  \begin{lem}\label{existence_xi} For any $2$-dimensional vector space $E \subset W_y$
  there exists $\xi \in \mathbb T$ killing all the forms in $E$. Hence, if $\omega_1,\omega_2$ is a basis of $E$, then $\xi \cdot \omega_1=\xi \cdot \omega _2=0$. 
  \end{lem}

We want to compute the adjunction form for a basis  $\omega_1, \omega _2$ of $E$, as defined in \cite{collino_pirola}. Observe that $\xi$ can be seen as an infinitesimal deformation $\mathcal A_{\Delta }$ of $A_y$. We denote by $F_\xi $  the rank $2$ vector bundle on $A_y$ attached to $\xi$ via the isomorphism $H^ 1(A_y,T_{A_y})\cong Ext^ 1(\Omega^1_{A_y},\mathcal O_{A_y})$. 
With this notation the sheaf $F_\xi$ can be identified with $\Omega ^1_{\cA_{\Delta}\vert A_y}$.
By definition there is a short exact sequence of sheaves:
  \[
  0\lra \mathcal O_{A_y} \lra  \Omega ^1_{\cA_{\Delta}\vert A_y} \lra \Omega^1_{A_y} \lra 0.
  \]
The connection map $H^0(A_y,\Omega^1_{A_y})\lra H^1(A_y,\mathcal O_{A_y})$ is the cup-product with $\xi \in H^1(A_y,T_{A_y})$.
Then the forms $\omega_1, \omega_2$  lift to sections $s_1,s_2\in H^ 0(A_y,\Omega ^1_{\cA_{\Delta}\vert A_y})$. 
These sections are not unique, but they are by imposing them to be $0$ on the $0$-section of $\mathcal A_{\Delta }\lra \Delta$.

Then the adjoint form of $\omega_1, \omega_2$ with respect to $\xi$ is defined as the restriction of 
\[
s_1\wedge s_2 \in H^0(A_y, \Omega ^2_{\cA_{\Delta}\vert A_y} )^0
\]
to $C_y$. This is a section of $H^0(C_y,\Omega ^2 _{\cC_{\Delta}|C_y} )\cong H^0(C_y,\omega _{C_y})$. 
 
\vskip 3mm
\begin{prop}
 \label{vanishing_implies_end} If the adjoint form vanishes then $\xi$ belongs to the kernel of $d\Phi: \mathbb T \lra T_{\cA_g}(A_y)=Sym^2 H^{1,0}(A_y)^*$.
\end{prop}

\begin{proof} 
 According to Theorem 1.1.8 in \cite{collino_pirola}, the adjoint form vanishes if and only if the image of $\xi $ in 
\[
H^1(C_y, T_{C_y}(D))\cong Ext^1(\omega_{C_y}(-D), \mathcal O_{C_y}) 
\]
 is zero. This says that the corresponding extension is trivial, so the short exact sequence in the first row of the next diagram splits (i.e. $i^* \Omega ^1_{\cC_{\Delta}\vert C_y}=\mathcal O_{C_y} \oplus \omega _{C_y}(-D)$):
 \[
  \xymatrix@C=1.pc@R=1.8pc{
 0 \ar[r] &
 \mathcal  O_{C_y} \ar[r] \ar@{=}[d]   &
i^* \Omega ^1_{\cC_{\Delta}\vert C_y} \ar[r] \ar[d]   &
 \omega _{C_y}(-D) \ar[r] \ar@{^{(}->}[d]^{i} &  0 \\
 0 \ar[r]&   O_{C_y} \ar[r]   &  \Omega ^1_{\cC_{\Delta}\vert C_y}  \ar[r]  & \omega _{C_y} \ar[r]  & 0
 }
 \]
which implies that the connecting homomorphism in the associated long exact sequence of cohomology $H^0(C_y,\omega _{C_y}(-D))\lra H^1(C_y,\mathcal O_{C_y})$ is trivial. Therefore $\xi \cdot H^0(C_y,\omega _{C_y}(-D))=0$ and in particular $\xi \cdot W_{y}=0$. Hence $\xi$ is in the kernel of $d\Phi_y$.  \end{proof}

\vskip 3mm
\textbf{End of the proof of \ref{m.res}:}
Since $d\Phi$ is injective in a generic point, we are reduced to prove the vanishing of the adjoint form to reach a contradiction. Our aim is to use the vanishing obtained in Corollary (\ref{vanishing_Delta}).

We fix a generic point $y\in \mathcal U$ and consider $(\xi, E)$ as in Lemma (\ref{existence_xi}).
As in the section 3, set $\Delta:=Spec \mathbb C[\varepsilon]/(\varepsilon ^2)$ and let $\alpha_{\xi } \lra \mathcal U $ be the map attached to $\xi $.
From now on we restrict our family over $\mathcal U$ to a family over $\Delta $.  Moreover we denote by $\cC_{\Delta }$ the pull-back  of $\mathcal C$ to $\Delta$, hence  we have an infinitesimal deformation of $C_y$:
 \[
 \cC_{\Delta } \lra \Delta.
 \]
Again we pull-back to $\Delta $ the family of abelian varieties and the family of curves  we get the diagram:
\begin{equation}\label{definition_Gamma}
  \xymatrix@C=1.pc@R=1.8pc{
 &\mathcal A_{\Delta } \ar[dd] \ar[rr] && \mathcal A \ar[dd]^{\pi} \\
  \cC_{\Delta}  \ar[ru]^{f_{\Delta }} \ar[rd] \ar[rr] && \mathcal {C} \ar[ru]^f \ar[rd]& \\
 &\Delta \ar[rr] &&\mathcal U. 
 }
 \end{equation}

Notice that $E$ is generated by two linearly independent forms $\omega_1, \omega_2 \in H^0(A_y,\Omega ^1_{A_y})\subset H^0(C_y,\omega_{C_y})$,  we still denote by $s_1,s_2\in H^ 0(A_y,\Omega ^1_{\cA_{\Delta}\vert A_y})$ the lifting of both sections. Then by the Proposition (\ref{vanishing_Delta}) the restriction to $C_{\Delta}$ of the form $s_1 \wedge s_2$  is zero. Hence the adjoint form is zero and thus the Theorem is proved. 
\qed

\vskip 3mm
\begin{rem}
 Since the Voisin set $V_2$ can be seen as the preimage of the rational orbit of the image of the origin in the corresponding Kummer variety our Theorem gives a bound on the dimension of the locus of Kummer varieties where these orbits are positive dimensional and ``non-degenerate'' (that is, the preimage in the abelian variety generates the abelian variety itself).
 \end{rem}

\end{document}